\documentclass[11pt]{amsart} 
\usepackage{amscd,amssymb, amsmath,amsthm}
\usepackage{color}
\usepackage{colordvi}
\newtheorem{theorem}{Theorem}[section]
\newtheorem{lemma}[theorem]{Lemma}

\newtheorem{corollary}[theorem]{Corollary}
\newtheorem{proposition}[theorem]{Proposition}
\theoremstyle{definition}

\newtheorem{definition}[theorem]{Definition}
\newtheorem{example}[theorem]{Example}
\newtheorem{examples}[theorem]{Examples}

\def\bc{\begin{center}}
\def\ec{\end{center}}

\def\rad{\operatorname{Rad}}
\def\Mod{\operatorname{Mod}}
\def\Ext{\operatorname{Ext}}

\def\End{\operatorname{End}}
\def\rad{\operatorname{rad}}
\def\gr{\operatorname{gr}}
\def\Hom{\operatorname{Hom}}
\def\rk{\operatorname{rk}}

\begin{document}

\title{Approach to artinian algebras via natural quivers}
\author{Fang Li}
\address{Department of Mathematics, Zhejiang University,
Hangzhou, Zhejiang 310027, China}
\email{fangli@cms.zju.edu.cn}
\thanks{Project supported by the National Natural Science
Foundation of China (No. 10871170) and the Natural Science
Foundation of Zhejiang Province of China (No. D7080064)}
\author{ Zongzhu Lin} \address{Department of Mathematics,
Kansas State University, Manhattan, KS 66506, USA}
\email{zlin@math.ksu.edu}
\thanks{ The second author is supported in part by an NSA grant and NSF I/RD program}

\maketitle

\begin{abstract}
Given an Artinian algebra $A$ over a field $k$, there are
several combinatorial objects associated to $A$. They are the
diagram $D_A$ as defined in [DK], the natural quiver $\Delta_A$
defined in \cite{Li} (cf. Section 2), and a generalized version of
$k$-species $(A/r, r/r^2)$ with $r$ being the Jacobson radical of
$A$. When $A$ is splitting over the field $k$, the diagram $D_A$
and the well-known ext-quiver $\Gamma_A$ are the same. The main
objective of  this paper is to investigate the relations among these
combinatorial objects and in turn to use these relations to give a
characterization of the algebra $A$.

\end{abstract}

\section{Introduction}

\subsection{} Given an Artinian algebra $A$ over a field $k$, there are
several combinatorial objects associated to $A$. They are the
diagram $D_A$ as defined in [DK], the natural quiver $\Delta_A$
defined in \cite{Li} (cf. Section 2), and a generalized version of
$k$-species $(A/r, r/r^2)$ with $r$ being the Jacobson radical of
$A$. When $A$ is splitting over the field $k$, the diagram $D_A$
and the well-known ext-quiver $\Gamma_A$ are the same. The main
objective of  this paper is to investigate the relations among these
combinatorial objects and in turn to use these relations to give a
characterization of the algebra $A$.


%


For a given Artinian $k$-algebra $A$, let $\{S_{1},S_{2},\cdots, S_{s}\}$ be the complete set of
non-isomorphic irreducible $A$-modules. Set $D_i=\End_{A}(S_i)$ which is
a division ring  and $ \Ext^1_{A}(S_i,S_j)$ is a
$D_i$-$D_j$-bimodule.  $A$ is said to
{\em split} over the ground field $k$, or say, to {\em be $k$-splitting}, if
 $\dim_{k}\End_A(S_{i})=1$ (i.e., $D_{i}=k$ for all irreducible $A$-modules $S_{i}$.  Recall that a quiver is a finite directed graph $\Gamma=(\Gamma_{0}, \Gamma_{1})$ with vertex set $ \Gamma_{0}$ and arrow set $ \Gamma_{1}$.  For a $k$-splitting algebra $A$, one can define a finite quiver $\Gamma_A$  called the {\em Ext-quiver}
of $A$ by setting $\Gamma_{0}=\{1,2,\cdots,
  s\}$, and $m_{ij}=\dim_k(\Ext^1_A(S_{i},S_{j}))$ being the number of   arrows from $i$ to $j$.



 There is another to characterize the Ext-quiver for a $k$-splitting algebra Artinian algebra.  By \cite{A} and \cite{Liu}, when $A$ is a finite-dimensional algebra over
  a field $k$ and $1=\varepsilon_{1}+\cdots
  +\varepsilon_{s}$ a decomposition of $1$ into a sum of primitive
  orthogonal idempotents, then, we can re-index $\{S_{1},S_{2},\cdots,
  S_{s}\}$ such that $S_{i}\cong A\varepsilon_{i}/r\varepsilon_{i}$ where $r$ is the radical of $A$,
  and moreover,
  $\dim_{k}\Ext_A(S_{i},S_{j})=
  \dim_{k}(\varepsilon_{j}r/r^{2}\varepsilon_{i})$.

If $ A$ is basis and $k$-splitting, then $A$ is a quotient of the path algebra of $\Gamma_{A}$.
The properties of this quiver $\Gamma_A$, in particular, its relation to the representations of
$ A$ has been extensively studied in the field of representations of
algebras. $\Gamma_A$  is invariant under Morita
equivalence, i.e., if $A$ and $B$ are two Morita equivalent
$k$-algebras, then $\Gamma_A$ is isomorphic to $ \Gamma_B$.

If $A$ is not basic,  it is not longer isomorphic
to a quotient of the path algebra $k \Gamma_A$. It is discussed in this paper and \cite{Li} how to get the analogue of the Gabriel theorem in this case.

\subsection{} In recent years, geometric methods has been heavily used in
representation theory of algebras. To each finite dimensional algebra
$A$ over an algebraically closed field $k$, one can associate a
sequence of algebraic varieties $ \Mod^d_{k}(A)$ ($d=1, 2, 3
\dots,$) as closed subvarieties of the affine spaces $k^{d\times d}$.
The association of the varieties depends on the presentation of the
algebra $A$ using finitely many generators and finitely many relations.
In \cite{bongartz}, it is proved that two algebras $A$ and $B$ are isomorphic if
and only if the associated varieties $ \Mod^d_k(A)$ and
$\Mod^d_k(B)$ are isomorphic as $GL_d(k)$-varieties.  Thus
having a more standard presentation of the algebra $A$ will help with
studying these varieties. The purpose of these paper is to explore
relations of the Ext-quiver of $A$ and the natural quiver which will be
defined in Section 2. The natural quiver will have fewer arrows than
the Ext-quiver when the algebra $A$ is not basic. Natural quivers are
not invariant under the Morita equivalence and much
 closer to reflect the structure of the algebra,
 rather than just its module category. There are numerous
 cases even in the representation theory that one
 needs the structure of the algebras, for example,
 the character values of finite groups in a block cannot be
 preserved through Morita equivalence.

\subsection{} The paper is organized as follows. In Section 2 we
 recall the definition of the natural quiver  $\Delta_A$ of an Artinian
 $k$-algebra $A$ and provide a precise relation with the
Ext-quiver $\Gamma_A$ when the algebra $A$ is splitting over the
ground field $k$. In Section 3, we prove in Theorem~\ref{th3.3} that
any Artinian algebra, which is splitting over its radical, is a quotient of
the generalized path algebra of its natural quiver associated to $A/r$. This gives a presentation of the algebra $A$.
Although there is always a surjective algebra homomorphism from the
path algebra of the natural quiver $\Delta_A$ to the tensor algebra $
T_{A/r}(r/r^2)$,  the above surjective map to $A$ does not always
factor through $T_{A/r}(r/r^2)$ (Example 3.5).  There have been
numerous generalizations of Wedderburn-Malcev theorems
to characterize an Artinian algebra that is splitting over its radical.  By
using the generalized path algebra of the natural quiver, we give another characterization of an Artinian
algebra $A$ which is splitting over its radical, see Corollary
\ref{cor3.4}. Moreover, we discuss the relations among the natural
quiver and the Ext-quiver of an Artinian algebra and the associated
generalized path algebra, see the figure in the end of Section 4, and
furthermore, their relationship with the diagram of an Artinian algebra
as defined in  \cite{DK}. The main
results are the formulae in Theorems \ref{th2.2} and \ref{th5.3}. As
an application, in Section 5, we discuss the relationship between the diagram and natural quiver of an artinian algebra that is a not splitting over the ground field. In Section 6 we prove that a (not necessarily basic) hereditary Artinian algebra which is splitting over its radical is isomorphic to the generalized path algebras of its natural quiver provided  the defining ideal as described in  Theorem~\ref{th3.3} does not interest with the arrow space (see Theorem \ref{theorem6.3}).

{\bf Acknowledgement.}\; The authors take this opportunity to
express thanks to B.M.Deng, M.M.Zhang, Y.B.Zhang and H.Y.Zhu
for their helpful conversations and suggestions.

\section{The relation between natural quiver and Ext-quiver}
Suppose that $A$ is a left Artinian $k$-algebra, and $r=r(A)$ is its
Jacobson radical.
%

 Write $A/r=\bigoplus_{i=1}^{s}{A}_{i}$ where
$ {A}_{i}$ are two-sided simple ideals of $A/r$. Such
decomposition of $ A/r$ is also called block decomposition of the
algebra $ A/r$. Then, $r/r^2$ is an
$A/r$-bimodule. Let $\;_{i}M_{j}={A}_{i}\cdot r/r^{2}\cdot
 {A}_{j}$ which is finitely generated as
$ {A}_{i}$-$ {A}_{j}$-bimodule for each pair
$(i,j)$.

For two rings $A$ and $B$,  and   a finitely generated
$A$-$B$-bimodule $M$, define $ \rk_{A, B}(M)$ to be the minimal number of generators
of $M$ as a $A$-$B$-bimodule among all generating sets.  As a convention, we always denote $ \rk_{A,B}(0)=0$.

 The isomorphism classes of
irreducible  $A$-modules is indexed by the set  $\Delta_{0}=\{1,\cdots,s\}$ corresponding to the set of blocks of $ A/r$. We now define the natural quiver $\Delta_{A}=(\Delta_{0}, \Delta_{1}) $ with $ \Delta_{0}$ being the vertex set and, for
$i,j\in \Delta_0$,  $t_{ij}=\rk_{A_{j},A_{i}}({}_jM_i)$ being the number of arrows from $i$ to
$j$ in $\Delta_{A}$. Obviously, there
is no arrow from $i$ to $j$ if $_{j}M_{i}=0$. This quiver $\Delta_A=(\Delta_{0},\Delta_{1})$ is
called the {\em natural quiver} of $A$.

The notion of natural quiver was first introduced in \cite{Li}, where
the aim  was to use the generalized path algebra from the
natural quiver of an Artinian algebra $A$ to characterize $A$ through
the generalized Gabriel theorem.
The advantage of
generalized path algebra is that valued quiver information is already
encoded in the generalized path algebras. In the language of
Kontsevich and Soilbelman \cite{KY}, Gabriel type algebra cannot
be stated as an affine non-commutative scheme which can be
embedded into a thin scheme in a sense that they are ``infinitesimally''
isomorphic. Result of this paper will be to find a ``smallest''
embedding.


For  a quiver $Q=(Q_{0}, Q_{1})$, a sub-quiver $Q'$  of $Q$ is called {\em dense} if  $(Q')_0=Q_0$
and for any vertices $i,j$,  there exist an arrow from $i$ to $j$ in
$Q'$ if and only if there exist an arrow from $i$ to $j$ in $Q$.

 When $A$ is splitting over the ground field $k$, then $ A_i\cong
M_{n_i}(k)$ and the irreducible module $S_{i}$ has $k$-dimension $n_{i}$. In this case the Ext-quiver $\Gamma_{A}$ of $A$ is defined (cf. 1.1).  It is proved in
  \cite[Prop. 7.4.3]{Liu} that
$$t_{ij}\leq m_{ij}\leq n_in_jt_{ij}.$$
In addition, if $A$ is basic, then $\Delta_{A}=\Gamma_{A}$ as discussed in the 1.1
 In general these two constructions will give two different
 quivers  if $A$ is not basic. The following results were proved in  \cite{LC1}.

  \begin{proposition}
Let $A$ be a $k$-splitting Artinian $k$-algebra over a field $k$ and  $B$ is the corresponding basic algebra of $A$.  Then,
\begin{itemize}
\item[(i)] $\Gamma_A=\Gamma_B$;

\item[(ii)]  $\Gamma_B=\Delta_B$;

\item[(iii)]
 $\Delta_A$ is a dense sub-quiver of $\Gamma_A$, and thus a dense
 sub-quiver of $\Gamma_B$ and $\Delta_B$.
\end{itemize}
  \end{proposition}

Now, we will just give the exact relation between $t_{ij}$ and
$m_{ij}$ when the algebra $A$ is splitting over $k$. First, we
recall that for any real number $a$, the ceiling of $a$ is defined
to be
$$\lceil a \rceil =\min\{ n \in \mathbb{Z} \;|\; n\geq a\}.$$

\begin{theorem}\label{th2.2}
Let $A$ be an Artinian $k$-algebra $A$ which is splitting over  $k$.
Assume $\Gamma_A$ is the Ext quiver of $A$ and $\Delta_A$ is
the natural quiver of $A$.  Then $t_{ij}=\lceil \frac{m_{ij}}{n_in_j}\rceil $ for $i,j\in (\Gamma_A)_0=(\Delta_A)_0$. Here
 $m_{ij}$ is the number of arrows from $i$ to $j$ in
$\Gamma_A$.
\end{theorem}
\begin{proof}
The proof involves computing a minimal generating set of
$_iM_j=A_i\cdot  (r/r^2 )\cdot A_j$ as
${A}_i$-$A_j$-bimodules. We first note that both
$A_i$ and $A_j$ are simple $k$-algebras and split over
$k$. Hence $A_i\otimes _k A_j^{op}$ is a also a simple
central $k$-algebra, isomorphic to $M_{n_in_j}(k)$. Since $r/r^2$ is
a semisimple left  $A_i$-module and  a semisimple right
$A_j$-module,  $_iM_j$ is a semisimple $A_i\otimes
_k A_j^{op}$-module with simple components isomorphic to $
S_i\otimes_{k} S_{j}^{op}$. Here $ S_{j}^{op}$ is the right irreducible
$A_j$-module. Let $ e_{kl}^{i}$ be the standard matrix basis
elements of $A_i=M_{n_i}(k)$. Then $\dim_{k}\Ext_{A}^1(S_i,
S_j)=\dim_{k}e_{11}^i r/r^2 e_{11}^{j}$. We claim that $_iM_j$ has
exactly $\dim_{k}e_{11}^i r/r^2 e_{11}^{j}$ many
$A_i\otimes_{k} A_j^{op}$-composition factors. Indeed  $
e_{11}^i\otimes e_{11}^j$ is a primitive idempotent of
$A_i\otimes_{k}A_j^{op}$. Now the theorem  follows from
the following lemma for simple Artinian rings.
\end{proof}
\begin{lemma} Let $R=M_n(D)$ be the ring of all $n\times n$-matrices
with entries in
 a division ring $D$ and $M$ be an $R$-module.
 Let $L=D^{n}$ be the natural irreducible $R$-module of column vectors.
 Then
 \begin{itemize}
\item[(i)] $M$ is a semisimple $R$-module isomorphic to
$L^{\oplus m}=L\oplus L\oplus \cdots \oplus L$, where
$m=\dim_{D}(e_{11}M)$;
\item[(ii)]  $M$ can be generated by $\lceil \frac{m}{n} \rceil $ many
elements over $R$, but cannot be generated by fewer number of elements.
    \end{itemize}
\end{lemma}

\begin{proof} Since $R$ is simple, all finitely generated $R$-modules
 are semisimple. Note that every irreducible $R$-module is isomorphic to
$L=D^n$.  Then (i) follows from the fact that $ \dim_{D} e_{11}L=1$.
To show (ii), we note that  $R=\oplus_{s=1}^n Re_{ss}\cong L^{\oplus
n}$ as left $R$-module. If $M$ is generated by $t$-many elements,
then $M$ is a quotient of $R^{\oplus t}$ as a left $R$-module.
Hence, $M$ has at most $tn$ composition factors counting
multiplicity, i.e., $m\leq tn$ and $\lceil \frac{m}{n} \rceil\leq
t$. Thus $M$ cannot be generated by less than $\lceil \frac{m}{n}
\rceil$ many elements. On the other hand, let $N$ be any $R$-module
of length $p\leq n$. Then $N$ is isomorphic to $L^{\oplus p}$.
There is a surjective homomorphism $ \phi: R\rightarrow N$ of left
$R$-modules. By writing $m=tn+s$, with $0\leq s < n$, we have
$\lceil \frac{m}{n} \rceil=t $ if $s=0$ and $\lceil \frac{m}{n}
\rceil=t+1$ if $s>0$. Hence we can construct a surjective
homomorphism $R^{\lceil \frac{m}{n} \rceil}\rightarrow L^{\oplus
tn+s}=(L^{\oplus n})^{\oplus t}\oplus L^{\oplus s}$. Hence $M$ is
generated by $\lceil \frac{m}{n} \rceil$ many elements.
\end{proof}

As we note in Proposition 2.1 (ii), the Ext-quiver and the natural
quiver of a finite dimensional  basic algebra coincide each other.
As an application of Theorem \ref{th2.2}, we give an example which
means the coincidence is also possible to happen for some non-basic
algebras.

\begin{example} $\;$ Let $k$ be a field of
 characteristic different from $2$ and let $Q$ be the quiver:
 \begin{figure}[hbt]
\begin{picture}(50,10)(80,-10)
\put(100,0){\makebox(0,0){$ \bullet$}}
\put(102,-10){\makebox(0,0){$e_1$}}
 \put(105,0){\vector(1,0){30}}
\put(95,0){\vector(-1,0){30}}
\put(140,0){\makebox(0,0){$\bullet$}}
\put(144,-10){\makebox(0,0){$e_{2'}$}}
\put(145,0){\vector(1,0){30}}
\put(180,0){\makebox(0,0){$\bullet$}}\put(190,0){\makebox(0,0){$e_{3'}$}}
 \put(60,0){\makebox(0,0){$\bullet$}}
 \put(62,-10){\makebox(0,0){$e_2$}}
\put(55,0){\vector(-1,0){30}}\put(40,5){\makebox(0,0){$\beta$}}\put(85,5)
{\makebox(0,0){$\alpha$}}
\put(120,5){\makebox(0,0){$\alpha'$}}\put(160,5){\makebox(0,0){$\beta'$}}
\put(20,0){\makebox(0,0){$\bullet$}}\put(12,0){\makebox(0,0){$e_3$}}
\end{picture}
\end{figure}
\end{example}

Let $\Lambda=kQ$ be the path algebra  of $Q$ and
$G=\langle\sigma\rangle$ be the automorphism group of $Q$
of order $2$. Then
$\sigma$ defines a $k$-algebra automorphism of $kQ$.   Now, we
consider the Ext-quiver and the natural quiver of the skew group
algebra $\Lambda G$ (see \cite A).

Let $r$ be the Jacobson radical of $\Lambda$. By Proposition 4.11 in \cite A,
$r \Lambda G)$ is the Jacobson radical of $\Lambda G$. It is easy to see that $(\Lambda
G)/(r\Lambda G)\cong (\Lambda/r)G$. In Page 84 of \cite A, it was
given that $(\Lambda/r)G\cong A_1\times A_2\times A_3\times
A_4=k\times k\times\left(
\begin{array}{cl}
k & k \\
k & k
\end{array}\right)\times\left( \begin{array}{cl}
k & k \\
k & k
\end{array}\right)$ as algebras and the associated basic
algebra $B$ is obtained in the reduced form from $\Lambda G$,
which is Mortia-equivalent to $\Lambda G$, and moreover, it was
proved in \cite A that $B$ is isomorphic to the path algebra of
the following quiver

\bigskip
\begin{center}
\begin{picture}(50,50)(80,0)
\put(100,0){\makebox(0,0){$ \bullet$}}
\put(100,-7){\makebox(0,0){$e_{(1)}$}}\put(60,-7){\makebox(0,0){$e_{(2)}$}}
 \put(105,0){\vector(1,0){30}} \put(100,35){\vector(0,-1){30}}
\put(100,38){\makebox(0,0){$\bullet$}}\put(105,20){\makebox(0,0)
{$\nu$}}\put(100,48){\makebox(0,0){$e_{(3)}$}}
\put(65,0){\vector(1,0){30}} \put(140,0){\makebox(0,0){$\bullet$}}
\put(140,-7){\makebox(0,0){$e_{(4)}$}}
 \put(60,0){\makebox(0,0){$\bullet$}}\put(80,7){\makebox(0,0)
 {$\lambda$}}\put(120,7){\makebox(0,0){$\mu$}}
\end{picture}
\end{center}
\bigskip

This quiver is just the Ext-quiver $\Gamma_{\Lambda G}$ of $\Lambda
G$. Therefore, all $m_{ij}=0$, or $1$.  For $i=1,2,3,4,\;\dim_kA_i=n_i^2$ where $n_1=n_2=1$, $n_3=n_4=2$.
By Theorem \ref{th2.2}, $t_{ij}=\lceil\frac{m_{ij}}{n_in_j}\rceil$. Then, for
each pair $(i,j)$, we have $t_{ij}=m_{ij}=0$ or 1. Therefore, the
natural quiver $\Delta_{\Lambda G}$ is equal to the Ext-quiver
$\Gamma_{\Lambda G}$.

\section{Algebras splitting over radicals}


The concept of generalized path algebra was introduced  early in
\cite{CL}. Here we review a different but equivalent definition.

Given a quiver $Q=(Q_{0}, Q_{1})$ and a collection of $ k$-algebras $ \mathcal{A}=\{ A_i \;|\; i \in Q_0\}$, let $e_i \in A_i $ be the identity.  Let $ A_0=\prod_{i\in Q_0}A_i$
be the direct product $k$-algebra. Note that
$e_i$ are orthogonal  central idempotents of $A_0$.

For $ i,j \in Q_{0}$, let $ \Omega(i,j)$ be the subset of arrows in $ Q_{1}$ from $i$ to  $j$. Define $ A_i \Omega(i,j)A_j$ to be the free $A_i$-$A_j$-bimodule (in the category of $k$-vector spaces) with
basis $\Omega(i,j)$. This is the free $ A_i\otimes_{k}
A_j^{op}$-module over the set $\Omega(i,j)$.

Then $M=\oplus_{i,j}  A_{i}\Omega(i,j)A_j$ is an $A_0$-$A_0$-bimodule.
The generalized path algebra is defined to be the tensor algebra
\begin{equation}\label{first} T_{A_0}(M)=\oplus_{n=0}^{\infty} M^{\otimes_{A_0}n}.\end{equation}
Here $M^{\otimes_{A_0}n}=M\otimes_{A_0}M\otimes_{A_0}\cdots
\otimes_{A_0}M$ and $ M^{\otimes_{A_0}0}=A_0$. We denote the
generalized path algebra by $ k(Q, \mathcal{A})$. Elements in
$M^{\otimes_{A_0}n}$ are called {\em virtual $\mathcal{A}$-paths}
of length $n$. In cases of path algebras, virtual
$\mathcal{A}$-paths are linear combinations of 
$\mathcal{A}$-paths of equal length. We denote
$J=\bigoplus^{+\infty}_{n=1}M^{\otimes_{A_{0}}n}$. Then $ k(Q, \mathcal{A})/J\cong A_{0}$.

The generalized path algebra has the following universal mapping
property.  For any $k$-algebra $B$ with any $k$-algebra homomorphism
$\phi_0: A_0 \rightarrow B$ (thus making $B$ an
$A_0$-$A_0$-bimodule) and any $A_0$-$A_0$-bimodule homomorphism $
\phi_1: M\rightarrow B$, there is a unique $ k$-algebra homomorphism
$\phi: k(Q, \mathcal{A})\rightarrow B$ extending $\phi_0$ and
$\phi_1$.

This definition is equivalent to the original definition in
\cite{CL} and has the advantage of the above mentioned universal
mapping property. As a matter of fact, the classical path algebras
are the special cases by taking $ A_i=k$. We are more interested
in the case when $A_i$ are simple $k$-algebras, in particular when
all $A_i$ are central simple algebras. The generalized path
algebra $k(Q, \mathcal{A})$ is called {\em normal} if all $A_i$
are simple $k$-algebras.

 A $k$-algebra $A$ is said to be {\em splitting over its radical} $r$
 if there is a $k$-algebra homomorphism $ \rho: A/r\rightarrow A$ such
 that $\pi\circ \rho =\operatorname{Id}_{A/r}$. For example if $A/r$ is separable and $A$ is Artinian, then $A$ is always splitting over $r$ as result of Wedderburn-Malcev theorem (and many generalizations in the literature \cite P).  Note that a normal generalized
  path algebra $k(Q, \mathcal A)$ with an acyclic (no oriented cycles) quiver $Q$ is always
  splitting over its radical $\rad(k(Q, \mathcal A))=J$. This equality fails in general for normal generalized path algebras.

\begin{proposition}
Let $Q$ be a finite  quiver and $k(Q,\mathcal A)$ be a normal
generalized path algebra with each $A_i$ being finite dimensional.
If  $I$ is an ideal of $k(Q, \mathcal{A}$ with $
J^{s}\subseteq I\subset J$ for some positive integer $s$. Then,  $k(Q,\mathcal
A)/I$ is a finite dimensional algebra and $\rad(k(Q,\mathcal A)/I)=J/I$.
\end{proposition}
\begin{proof} Since $Q$ is finite  quiver and $ A_{i}$ are finite dimensional, we have $ M^{\otimes _{A_{0}}n}$ in (\ref{first}) is finite dimensional over $k$ for all $n$.   Hence $k(Q, \mathcal{A})/J^{s}$ is also finite dimensional.
  It follows that $k(Q,\mathcal A)/I$ is
finite dimensional. On the other hand, $J/I$ is a nilpotent ideal of $k(Q, \mathcal
A)/I$ with $(J/I)^s=0$ and $(k(Q,\mathcal A)/I)/(J/I)\cong
k(Q,\mathcal A)/J\cong\bigoplus_{i\in Q_0}A_i $ is a semisimple
algebra. This implies $\rad(k(Q,\mathcal A)/I)=J/I$.
\end{proof}

We remark that the finite dimensionality of $A_i$ over $k$ cannot be
removed. For example, take the Dynkin quiver $Q$ of type $A_2$
with two vertices $\{ 1,2\}$ and one arrow $ \alpha$  from 1 to 2.
Let $A_1 $ and $ A_2$ be two infinite dimensional field extensions
of $k$. Although $A_1$ and $ A_2$ are two simple $k$-algebras,
but the path algebra $k(Q, \mathcal{A})$ is not left or right Artinian
since
$$ A_2\alpha V=\{\sum_ia_i\alpha b_i: a_i\in A_2, b_i\in V\}$$
 is a left ideal of $k(Q, \mathcal{A})$ for any vector subspace $V$ of
$A_1$.

In this section we will show that every Artinian algebra $A$ which
is splitting  over its radical will be a
 quotient of a generalized path algebra of its natural quiver.

For an Artinian algebra $A$ with radical $r$, let $\Delta_A$ be its
natural quiver and $A/r=\bigoplus_{i\in(\Delta_A)_0}A_i$ where all
$A_i$ are simple algebras. Denote $\mathcal A=\{A_i:\;
i\in\Delta_0\}$. Then, the generalized path algebra
$k(\Delta_A,\mathcal{A})$ is called {\em the associated generalized
path algebra} of $A$.

In \cite{LC1}, we introduce the following notion of Gabriel-type
algebra.
\begin{definition}\label{3.2}
Let $A$ be an Artinian $k$-algebra  and
$k(\Delta_A,\mathcal A)$ be its associated  generalized path
algebra. $A$ is said to be of
{\em Gabriel-type for generalized path algebra} if $A\cong k(\Delta_A,\mathcal A)/I$ for some ideal $I$ of $k(\Delta_A,\mathcal A)$ contained in $J$. We call $I$ the defining ideal of $A$.
\end{definition}

In \cite{Li} and \cite{LC1}, the conditions for certain artinian algebras to be Gabrial type were discussed under different assumptions. Here we give a more general description.

%


\begin{theorem}\label{thm3.3}
Let $A$ be a Gabriel-type Artinian algebra for generalized path
algebra over a field $k$ such that $A\cong k(\Delta_A,\mathcal A)/I$
with the ideal $I$ of $k(\Delta_A,\mathcal A)$ satisfying
$I\subset J$. Then, $A$ is
splitting over its radical $r$.
\end{theorem}

\begin{proof}  The assumption $ I \subseteq J$ implies
$$(k(\Delta_A,\mathcal{A})/I)/(J/I)\cong
k(\Delta_A,\mathcal{A})/J\cong A_{1}\oplus \cdots \oplus A_{p}\cong
A/r$$ is semisimple, then $\rad A\cong
\rad (k(\Delta_A,\mathcal{A})/I)=J/I$. We have $A_{1}\oplus \cdots
\oplus A_{p}\subset k(\Delta_A,\mathcal{A})$ and $(A_{1}\oplus
\cdots \oplus A_{p})\cap I\subset (A_{1}\oplus \cdots \oplus
A_{p})\cap J=0$. Hence, we get $A_{1}\oplus \cdots \oplus
A_{p}\stackrel{\varepsilon}{\hookrightarrow}
k(\Delta_A,\mathcal{A})/I \stackrel{\eta}{\twoheadrightarrow}
(k(\Delta_A,\mathcal{A})/I)/(J/I) \cong
k(\Delta_A,\mathcal{A})/J\cong A_{1}\oplus \cdots \oplus A_{p}$
which implies $\eta\varepsilon=1$. Therefore, $A$ is splitting over
$r$.
\end{proof}


In fact, in Theorem 8.5.4 of \cite{DK}, it was proven that, for a
finite dimensional algebra $A$ with radical $r$, if the quotient
algebra $A/r$ is separable, then $A$ is isomorphic to a quotient
algebra of the tensor algebra $T_{A/r}(r/r^2)$ by an ideal $I$ such
that $J^s\subset I\subset J^2$ for some positive integer $s$.
 The
separability condition plays to two roles in the proof. First, it
guarantees the Wedderburn-Malcev theorem to get
\begin{itemize}
\item[(a)]
$ A$
 is splitting over its radical.
 \end{itemize}
Secondly, the seperability also implies that
\begin{itemize}
\item[(b)] $ r/r^2 $  isomorphic to a
direct summand of $r$ as an $A_0$-$A_0$-bimodule.
\end{itemize}
It turns out
 the properties (a) and (b) together is equivalent to the existence
of the surjective map $\phi$ in the theorem.

It is proved in \cite{Li} that there always exists a surjective homomorphism of
algebras $\pi: k(\Delta_A,\mathcal A)\rightarrow T_{A/r}(r/r^2)$. This can be seen from the universal property of $k(\Delta_{A}, \mathcal{A})$.  Hence any artinian algebra $A$ with separable quotient $A/r$ is
isomorphic to a quotient algebra of $k(\Delta_A,\mathcal A)$ by an
ideal $I$ as in Theorem \ref{thm3.3}.   The different point
in this paper and that in  \cite{Li}  the algebra $A$ is
splitting over its radical and the given ideal $I$  is included
in $J$ but not in $J^2$. The following theorem is an improvement of the result in \cite{Li}.

\begin{theorem} \label{th3.3} Let $A$ be an Artinian $k$-algebra such
that $A$ is splitting over its radical.
Then there is a surjective algebra homomorphism $ \phi:
k(\Delta_{A}, \mathcal{A})\rightarrow A$ with $ J^s\subseteq
\ker(\phi)\subseteq J$ for some positive integer $s$.
\end{theorem}

\begin{proof}
 Let $A/r=A_0=\oplus_{i\in \Delta_0}A_{i}$.  Then
 $A_0$ is semisimple and $r/r^2$ is a semisimple
 $A_0$-module. Since $A$ splits over $r$, we can
 regard $A_0$ as a subalgebra of $A$ and hence $r$ is
 an $A_0$-$A_0$-bimodule under the multiplication in
 $A$ and the quotient map $\bar{}: r\rightarrow r/r^2$
is a homomorphism of $A_0$-$A_0$-bimodules. We will used this
property in the following argument.

 For each pair $(i,j)$, let
  $_iM_j=A_i(r/r^2)A_j $. Let
  $X_{ij}=\{ f^{ij}_{l} \in A_irA_j
  \subseteq r\;|\; l=1,\dots, t_{ij}\}$ such that
  the image
  $\bar{X}_{ij}=\{\bar f_{l}^{ij}\;|\; l=1, \dots,t_{ij}\}$ in
  ${}_iM_j$
  is a minimal generating set as
  an $A_i$-$A_j$-bimodule. Let $X=\cup_{i,j\in\Delta_0}X_{ij}$.
 Then, its image $\bar{X}=\cup_{i, j \in \Delta_0} \bar{X}_{ij}$
 generates $r/r^2=\sum_{i, j \in \Delta_0}A_i(r/r^2)A_j$ as an
 $A_0$-$A_0$-bimodule.

 We now show that the subalgebra $S$ of $A$ generated by
 $A_0$ and the set $ X$
is actually $A$. Consider the associated graded algebra $
\gr(A)=\oplus_{s\geq 0} r^s/r^{s+1}$, with $r^0=A$. Then
$\gr(S)=\oplus_{s\geq 0}( S\cap r^s)/(S\cap
 r^{s+1})$ is the graded  subalgebra of $\gr(A)$
 generated by $A_0$ and
 $\bar{X}=\{ \bar{f}^{ij}_{l}\;|\; 1\leq l \leq
 t_{ij}, i, j \in \Delta_0\}$.

We claim that $ \gr(S)=\gr(A)$.  We will show that
$S_p=r^p/r^{p+1}$, where $S_p=(S\cap r^{p})/(S\cap r^{p+1})$ are the
homogeneous components of the graded algebras  $\gr{S}$. Note that
both $ S_p$ and $r^p/r^{p+1}$ are $A_0$-$A_0$-bimodules. It follows
from the definition that $S_1=(A_{0}X A_0)/(r^2\cap S)\subseteq A_0
\bar{X}A_0=r/r^2$.
 By the choice of the set $\bar X$, we have $A_0XA_0$ maps
 onto $ r/r^2$ (under the quotient map $ r\rightarrow r/r^2$).
  Hence, we have $S_1=r/r^2$.
The multiplication
\[\underbrace{r/r^2\otimes_{A_0}r/r^2\otimes_{A_0}\cdots
\otimes_{A_0}r/r^2}_{\text{$p$ times}}\rightarrow r^p/r^{p+1}
\]
in $\gr(A)$ is surjective following the definition of $r^p$. Thus \[
 r^p/r^{p+1} = \underbrace{S_1\cdots S_1}_{\text{$p$ times}}\subseteq
 S_p.\]
But we have $ S_p\subseteq r^p/r^{p+1}$. Hence $ S_p=r^p/r^{p+1}$
and $\gr(A)=\gr(S)$.

We now prove that $S=A$. Otherwise, we must have $r\nsubseteq S$
since $A$ is generated by $A_0$ and $r$. Let $p$ be maximal
 such that $ r^p\subseteq S$ is not true. Such $p$ exists since
$r^m=0$ for $m>>0$ ($A$ is Artinian) and $p\geq 1$. Take $ a \in
r^p \setminus S$ and  let $\bar{a}\in r^p/r^{p+1}$ be the image of
$a$. Since $S_p=r^p/r^{p+1}$, there is $b \in S\cap r^p$ such that
$ \bar{b}=\bar{a}$ in $r^p/r^{p+1}$. Thus $ a-b\in
r^{p+1}\subseteq S$ ($p$ is maximal). Hence, $a=b+(a-b)\in S\cap
r^p$. This is a contradiction. Hence we have $ S=A$.

Since $A$ splits over $r$ and $A_0$ is a subalgebra of $A$ such that
the quotient map $A\rightarrow A/r$ restricts to $A_0$ is the
identity map.  Then $A_0XA_0\subseteq A$ is an
$A_0$-$A_0$-sub-bimodule of $A$. Since
$|\Omega(i,j)|=t_{ij}=rk({}_iM_j)$, there is a surjective
homomorphism of  $A_0$-$A_0$-bimodules: $
A_0\Omega(i,j)A_0\rightarrow A_0 X_{ij}A_0$. Now by the universal
mapping property of the generalized path algebra, there is a
surjective algebra homomorphism $\phi: k(\Delta_A,
\mathcal{A})\rightarrow S$ since  $S$ is  generated by $A_0$ and the
set $X$ as a subalgebra of $A$. Hence
  $\phi(k(\Delta, \mathcal{A}))=S=A$ as we have just proved and
  $\phi\mid_{A_0}=\operatorname{id}_{A_0}$.

Let $I=\ker(\phi)$ and $J$ be the ideal of $k(\Delta_A,
\mathcal{A})$ generated by all virtual $\mathcal A$-paths of length
$1$. Then $\phi(J)= r$ and thus induces $k(\Delta, \mathcal{A})/J
\cong A/r$. Hence $ \ker(\phi)\subseteq J$. Since $\phi(J^p)=r^p$
for all $ p\geq 0$. We have $ \phi(J^s)=r^s=0$ for some $s>0$.
This completes the proof of the theorem.
\end{proof}

 Theorem 3.4
requires that the Artinian algebra $ A$ is splitting over its radical, i.e.,
the property (a) only without (b) as mentioned above, although the tenser algebra has to be
replaced by the generalized path algebra. In case $r_A^2=0$, then
the condition (b) is automatic provided (a) holds.

 For an Artinian algebra $A$, setting $A_0=A/r_A$ which is a semi-simple algebra
 with an $A_0$-$A_0$-bimodule structure on $M=r_A/r^2_A$. Then we
 can call the ordered pair $(A_0, r_A/r^2_A)$ a {\em generalized $k$-species}
 in the language of $k$-species discussed in \cite{R}. The tensor algebra
 $T_{A_0}(M)=\oplus_{n=0}^{\infty}M^{\otimes_{A_0}n} $ is an associate algebra
 (not necessarily  Artinian). The generalized path algebra of a quiver is
 naturally the tensor algebra of a generalized $k$-species. The properties of this path algebra is controlled by the bi-module structure of $M$. This algebra plays an important role in non-commutative geometry, which will be studied in the next paper.

 The following
example shows the difference between the generalized path algebra
and the tensor algebra of the associated generalized $k$-species in the fact
that $\phi$ does not exist with condition (a) only.

One also notes that the surjective algebra homomorphism $\pi:
k(\Delta_A, \mathcal{A})\rightarrow T_{A_0}(r_A/r^2_A)$
is an isomorphism if and only is $A_i (r_A/r^2_A)A_j$ is a free
$A_i\otimes_{k}A_j$ -bimodule for all $ i , j$.  One natural question
is, for a general $\pi$,  whether there is a map $ \psi: T_{A_0}(r_A/r^2_A)\rightarrow A$ such that $ \psi\circ \pi =\phi$. The following example gives an answer to
this question. It also shows that one cannot expect to generalize
\cite[Th. 8.5.2]{DK} to non-separable cases.

\begin{example} Let $ D=k[\alpha]$ with $k=\mathbb{F}_p(t)$ being the
transcendental extension of the finite field with $p$ elements and
$\alpha= \sqrt[p]{t}$. Then $D$ is a purely inseparable field
extension of  $k$. Now $ B=D\otimes_{k}D$ is not  semisimple
and its radical $r_B$ is nilpotent and generated by $
x=\alpha\otimes_{k}1-1\otimes_{k}\alpha$ as a
$D$-$D$-subbimodule of $B$. Let $ A=D\oplus r_B$ as a
$k$-vector space with multiplication defined by $(d, z)(d', z)=(dd',
dz'+zd'+zz')$ with $ d,\, d' \in D$ and $z,\, z' \in r_B$ . Here $dd'$
and $ zz'$ respectively the multiplications in $D$ and $ r_B$
respectively by noting the $D$-$D$-bimodule structure on $r_B$. This makes $A$ an associative algebra and
 $A/r_A=D$ is not separable.  But $ A$ is splitting over its
radical. $A$ is a not a commutative algebra (unless $p=2$) and $
r_A/r^2_A=r_B/r^2_B\cong D$ as $D$-$D$-bimodules with the
left and right actions of $D$ coincide. Note the left and right actions of
$D$ on $r_B$ are not the same unless $p=2$.  Hence the tensor
algebra $T_D(r_A/r^2_A)$ is a commutative algebra. But $A$ is not
commutative. Hence there is no surjective $k$-algebra map $\psi:
T_{D}(r_A/r^2_A)\rightarrow A$.
\end{example}

This example shows that under the assumption that $A$ is splitting
over its radical, the surjective map $\phi: k(\Delta_A,
\mathcal{A})\rightarrow A$ in Theorem 3.4 cannot be factored
through $ T_{A_0}(r_A/r^2_A)$.

As mentioned earlier that the separability condition on $A_0$ implies
that $ A$ is splitting over its radical, by Wedderburn-Malcev
theorem. There has been numerous generalizations of
Wedderburn-Malcev theorem in the literature. Combining Theorem
\ref{thm3.3} and Theorem 3.4, the following gives a characterization
of Artinian algebras that is splitting over its radical in terms of natural
quivers and  the associated generalized path algebras.

\begin{corollary}\label{cor3.4}
An Artinian $k$-algebra $A$  is splitting over its radical if and only if
$A$ is of Gabriel-Type for generalized path algebras (cf. Def. 3.2), i.e.,
there is a surjective algebra homomorphism $ \pi: k(\Delta_{A}, \mathcal{A})\rightarrow A$ such that  $\ker(\pi)\subseteq J$.
\end{corollary}


We will see in Section 6  it will be important if the condition $
 \ker(\pi)\subset J$  in Corollary \ref{cor3.4} is
replaced by the stronger one, that is, $ \ker(\pi)\subset J^2$. In this case, the ideal
$I=\ker(\pi)$ is said to be {\em admissible}.

\section {The relations among quivers arisen from an Artinian algebra
to its generalized path algebra}
 In this section, we use the relation between
 the Ext-quiver and the natural quiver in Section 2 to study the associated normal
 generalized path
 algebras of Artinian algebras. Note that the definition of the Ext-quiver
   always requires that an Artinian algebra $A$ is splitting over the
  field  $k$, which we will assume in this section.

  The natural quiver $\Delta_A$ of an Artinian algebra $A$ is
always finite, i.e. including finitely many vertices and finitely many arrows. In the sequel, we assume always that
$\Delta_A$ is acyclic (i.e., $\Delta_{A}$ does not have oriented cycles of length at least 1). Trivially, this is the sufficient and necessary
condition under which
 the associated  $k(\Delta_A,\mathcal A)$ is Artinian.

By definition, acyclicity of $ \Delta_{A}$ implies that the
natural quiver of $k(\Delta_A,\mathcal A)$ is just that of $A$, that is,
$\Delta_{k(\Delta_A,\mathcal A)}=\Delta_A$.  Denote by $m_{ij}$ and $g_{ij}$ the arrow multiplicities
respectively in the Ext-quivers $\Gamma_{A}$ and $\Gamma_{k(\Delta_A,\mathcal
A)}$ respectively. Note that they
should not be confused with each other. In general,
$ m_{ij}$ and $g_{ij}$ are quite different since the representation theories of $
A$ and $k(\Delta_A,\mathcal
A)$ are quite different (cf. Section 5).

For the radical $J$ of $k(\Delta_A,\mathcal A)$,
$k(\Delta_A,\mathcal A)/J\cong A/r_A=\bigoplus_{i\in\Delta_0}
A_i$. By Theorem 2.2, the natural quiver  $\Delta_{k(\Delta_A,\mathcal A)}$
is a sub-quiver of the Ext-quiver $\Gamma_{k(\Delta_A,\mathcal A)}$  and exactly,
 \begin{equation}
 t_{ij}=\lceil \frac{g_{ij}}{n_in_j}\rceil.
 \end{equation}


Denote $I=(\Delta_A)_0$. A {\em complete set of non-isomorphic primitive orthogonal
idempotents} of $A$ is a set of primitive orthogonal idempotents
$\{e_i:i\in I\}$ such that $Ae_{i}\not\cong
Ae_{j}$ as left $A$-modules for any $i\not=j$ in $I$ and for each
primitive idempotent $e$ the module $Ae$ is isomorphic to one
of the modules $Ae_i$ ($i\in I$).

Let
$\overline e_i=e_i+r_A$.  Then, by \cite{LC1}, $\{\overline e_i: i\in I\}$
is a complete set of non-isomorphic primitive orthogonal idempotents
of $A/r$. Let $ \tilde{e}_{i}\in k(\Delta_{A}, \mathcal{A}) $ be the lift of $ \bar{e}_{i}$ (we have assumed that $ k(\Delta_{A}, \mathcal{A})$ is artinian).
 Then, $S_i\cong Ae_i/r_Ae_i=(A/r_A)\overline e_i\cong (k(\Delta_A,\mathcal A)\tilde
{e}_i)/(J\tilde{e}_i)$, as $ i \in I$,  give a list of all non-isomorphic irreducible modules for both $A$ and  $k(\Delta_A,\mathcal A)$.

For $i\in I$, the identity $1_{A_{i}}$ of $A_{i}$ can be
decomposed into a sum of primitive idempotents, i.e.
$1_{A_{i}}=e_{11}^{i}+\cdots +e_{n_{i}n_{i}}^{i}$, and we can assume  $e_{11}^{i}=e_i$.

Note that $A$ is $k$-splitting. By (\cite{A}, Proposition III.1.14),
we have
$$g_{ij}=\dim_{k}\Ext^{1}_{k(\Delta_A,\mathcal A)}(S_{i},S_{j})=
\dim_{k}(\tilde{e}_{kk}^{j}J/J^{2}\tilde{e}_{ll}^{i})$$ for all $ k, l$.
Therefore,
\begin{eqnarray*}
\dim_{k}A_{j}J/J^{2}A_{i}
&=&\dim_{k}1_{A_{j}}J/J^{2}1_{A_{i}}
=\sum_{k=1}^{n_j}\sum_{l=1}^{n_i}\dim_{k}e_{kk}^{j}J/J^{2}e_{ll}^i\\
&=&n_{i}n_{j}\dim_{k}\Ext_{k(\Delta_{A},\mathcal{A})}^{1}(S_{i},S_{j})=n_{i}n_{j}g_{ij}.
\end{eqnarray*}
Recall that the number of arrows from $i$ to $j$ in
 $\Delta_{k(\Delta_A,\mathcal A)}=\Delta_A$ is $t_{ij}$. These
 $t_{ij}$ arrows generate freely the $A_j$-$A_i$-bimodule
 $A_{j}J/J^{2}A_{i}$. Thus,
 $$\dim_{k}A_{j}J/J^{2}A_{i}=\dim_{k}A_{j}
 \dim_{k}(J/J^{2})\dim_{k}A_{i}=n^2_jt_{ij}n^2_i.$$
 Hence, $g_{ij}=n_in_jt_{ij}$, that is, we obtain the following:
\begin{proposition}
 Let $A$ be a $k$-splitting finite dimensional algebra with radical $r$, whose natural
 quiver $\Delta_A$ is acyclic. Write $A/r
=\oplus_{i\in(\Delta_A)_0}  A_i$ where $A_i$ are simple
algebra for all $i$ with $n_i=\surd\overline{\dim_k  A_i}$.
   Then, the natural quiver of $k(\Delta_A,\mathcal A)$ is the
   same with that of $A$ and
   $$g_{ij}=n_in_jt_{ij}$$ where
$g_{ij}$ is the number of arrows from $i$ to $j$ in
 $\Gamma_{k(\Delta_A,\mathcal A)}$,  $t_{ij}$ is the number of
 arrows from $i$ to $j$ in
 $\Delta_A$.
\end{proposition}

For the number $m_{ij}$ of arrows from $i$ to $j$ in
 $\Gamma_{A}$, by Theorem 2.2,
 $t_{ij}=\lceil \frac{m_{ij}}{n_in_j}\rceil$. Then,
 $\lceil\frac{m_{ij}}{n_in_j}\rceil=\frac{g_{ij}}{n_in_j}$, equivalently, we
 have
\begin{corollary}
$m_{ij}\leq g_{ij}<m_{ij}+n_in_j$.
\end{corollary}


The set $\{P_i=Ae_{i}\;|\; i \in I\}$ is a complete set of
representatives of the iso-class of indecomposable projective $A$-module.
Then the basic algebra $B$ of $A$ is given by
$$B=\End_A(\coprod_{i\in I}P_i)\cong
\bigoplus_{i,j\in I}\Hom_A(P_i,P_j)
\cong\bigoplus_{i,j\in I}e_iAe_j.$$
Similarly the basic algebra of $k(\Delta_{A}, \mathcal{A})$ is
$$C=\End_{k(\Delta_A,\mathcal A)}(\bigoplus_{i\in I}k(\Delta_A,\mathcal
A)\overline e_i)\cong\bigoplus_{i,\,j\in I}\overline
e_ik(\Delta_A,\mathcal A)\overline e_j.$$ The relationship between
the two basic algebras $B$ and $C$ is given in \cite{LC1} when $A$
is of Gabriel-type.

As we have said, $k(\Delta_A,\mathcal A)$ is not Artinian when
$\Delta_A$ has an oriented cycle. Hence, we cannot affirm whether
$C$ is Morita equivalent to $k(\Delta_A,\mathcal A)$ in general.
But, $C$ is still decided uniquely by $k(\Delta_A,\mathcal A)$. So,
we call $C$ the {\em basic algebra associated to}
$k(\Delta_A,\mathcal A)$.

 As we have assumed  $\Delta_A$ is acyclic in this section,
  $C$ is Morita equivalent to
$k(\Delta_A,\mathcal A)$ since $k(\Delta_A,\mathcal A)$ is
Artinian. Hence, their Ext-quivers are the same, that is,
$\Gamma_{k(\Delta_A,\mathcal A)}=\Gamma_C$. And, since $C$
is basic, $\Gamma_C=\Delta_C$.

To sum up, assuming that the $ \Delta_{A}$ is acyclic,  we get the following diagram:

\[
\begin{array}{cclccclcc}
\Gamma_A&\stackrel{ t_{ij}=\lceil \frac{m_{ij}}{n_in_j}\rceil }{\supset}
&\Delta_A&=&\Delta_{k(\Delta_A,\mathcal A)}
&\stackrel{ t_{ij}=\frac{g_{ij}}{n_in_j}}{\subset}
&\Gamma_{k(\Delta_A,\mathcal A)}\\
\parallel &   &\cap &       &\cap &  &\parallel \\
\Gamma_B  & = & \Delta_B&\stackrel{m_{ij}\leq
g_{ij}<m_{ij}+n_in_j}{ \subset}  & \Delta_C & =& \Gamma_C
\end{array}
\]
where $\subset$,
$\supset$ and $\cap$ mean the embeddings of the dense sub-quivers.

\section{Diagram for non-splitting algebras}

In Section 2 and 4, the Artinian algebra $A$ is required to be splitting
over the ground field $k$ due to the definition of Ext-quiver. If
$A$ is not splitting over $k$, usually
 Ext-quiver and its representations have to be respectively replaced by
  the so-called {\em valued quiver} or {\em $k$-species}.
  On the other hand, the notion of the {\em diagram} of an Artinian algebra is
  introduced  in \cite{DK}  in the case when $A$ is not necessarily
  splitting over $k$.
 Now, we recall the definition of the diagram of an Artinian algebra
 $A$.

 Let $P_1, P_2, \cdots, P_s$ be pairwise non-isomorphic principal
 indecomposable projective modules over an Artinian algebra $A$, corresponding to the simple
 components $A_1,
  A_2, \cdots, A_s$  of the
  semisimple algebra $\overline{A}=A/r=\oplus_{i=1}^sA_i$.  Write $R_i=rP_i$, then $S_i=P_i/R_i$ is the
  corresponding irreducible $A$-module. At the same time,  $V_i=R_i/rR_i$ is a semisimple left $A$-module which
  has a direct sum decomposition  $V_i\cong
  \oplus_{j=1}^s S_j^{\oplus h_{ij}}$ for some unique integers $h_{ij}$. Define the quiver $D_{A}=(D_{0}, D_{1})$ by setting $D_{0}=\{1, \dots, s\}$ and  the arrow set $ D_{1}$ such that there are exactly $h_{ij}$
 many arrows from $i$
 to  $j$ for $ i, j \in D_{0}$. This quiver $D_{A}$ is
 called {\em the diagram of the algebra} $A$.

 Observe that:  (i)\ the projective cover  of $V_{i}$ is $P(R_i)\cong\oplus_{j=1}^s
  P_j^{\oplus h_{ij}}$ for each $i$; (ii)\  two Morita equivalent algebras
  have the same diagrams; (iii)\  the
  diagrams of $A$ and $A/r^2$ coincide.


\begin{proposition}  Let $S_1, S_2, \cdots, S_s$ be
pairwise non-isomorphic irreducible modules over an Artinian algebra
$A$. Then, in the diagram $D_A$ of $A$, the number of arrows
from the vertex $i$ to the vertex $j$ is
 $$ h_{ij}=\dim_{D_j}\Ext_{A}(S_i, S_j)$$
 where $D_j=\End_{A}(S_j)$ is a division algebra.
\end{proposition}
\begin{proof} Applying the functor $\Hom_A(-,S_j)$ to the short exact sequence $0\rightarrow R_i\rightarrow P_i\rightarrow S_i\rightarrow 0$,
one gets the  long
exact sequence, which gives the isomorphisms
$$\Ext^{p}_A(R_i,S_j)\cong \Ext^{{p+1}}_A(S_i,S_j)
\quad \quad (p\geq 0).$$
Since $R_{i}/rR_i=\oplus_{l=1}^s S_l^{\oplus h_{il}}$,
we have
$\Hom_A(R_i/rR_i,S_j)\cong D_j^{h_{ij}}$. Now the
proposition follows from
$\Hom_A(R_i/rR_i,S_j)\cong \Hom_A(R_i,S_j)$.
\end{proof}

When $A$ is $k$-splitting, each $D_j=k$, then $
h_{ij}=\dim_k\Ext^1_{A}(S_i, S_j)=m_{ij}$. Hence, in this case, the
Ext-quiver $\Gamma_A$ is equal to the diagram $D_A$, whose
relationship with $\Delta_A$ is given in Theorem \ref{th2.2}.
When $A$
is not $k$-splitting, in the place of Ext-quiver $\Gamma_A$
is a valued quiver with modulation. The following gives a
relation of the diagram $D_{A}$ with the valued quiver and
associated modulation as well as with the natural quiver $\Delta_{A}$
for a general Artinian algebra $A$ .

Using the earlier notations  $A/r=A_1\oplus\cdots\oplus A_s$ with
$A_i\cong M_{n_i}(D_i)$ (the matrix algebra with entries in $D_i$).
Let $e_{pq}^{i}$ be the matrix
basis element of $A_{i}$ with 1 at $(p,q)$ position and zero
anywhere else. The irreducible left $A_{i}$-module
$S_{i}\cong A_{i}e^{i}_{11}$ and irreducible right $A_{i}$-module
$S^{op}_{i}\cong e^{i}_{11}A_{i}$. The $A_{j}$-$A_{i}$-module
$A_{j}(r/r^{2})A_{i}$ is semi-simple both as left $A_{j}$-module
and  as right $A_{i}$-module.   Let $D_{l}$ act on $S_{l}$ from
right and $D_{l}$ act on $S^{op}_{l}$ from left. For any
$A_{j}$-$A_{i}$-bi-module $M$ which is semisimple as $A_{j}$-module and $ A_{i}$-module respectively, we have the following:
\begin{itemize}
\item[(a)] $\Hom_{A_{j}}(S_{j}, M)$ is a $D_{j}$-$A_{i}$-bi-module
isomorphic to $ e^{j}_{11}M$ and the map
$ S_{j}\otimes _{D_{j}}\Hom_{A_{j}}(S_{j}, M)\rightarrow M $
defined by $ x\otimes \phi=\phi(x)$  is an isomorphism of
$ A_{j}$-$A_{i}$-bi-modules;

\item[(b)] $\Hom_{A_{i}}(S^{op}_{i},\Hom_{A_{j}}(S_{j}, M))$ is a
$D_{j}$-$D_{i}$-bimodule isomorphic to $e^{j}_{11}Me^{i}_{11}$
and the map
$ \Hom_{A_{i}}(S^{op}_{i},\Hom_{A_{j}}(S_{j}, M))
\otimes _{D_{i}}S_{i}^{op}\rightarrow \Hom_{A_{j}}(S_{j}, M)$
defined by $ f\otimes v=f(v) $ is an isomorphism of
$D_{j}$-$A_{i}$-bi-modules.

\item[(c)] The natural multiplication map
$S_{j}\otimes_{D_{j}}e^{j}_{11}Me_{11}^{i}\otimes S_{i}^{op}
\rightarrow M$ is an isomorphism of $A_{j}$-$A_{i}$-bi-module.
\end{itemize}
Taking $M=A_{j}(r/r^{2})A_{i}$, we have
an $A_{j}$-$A_{i}$-bi-module isomorphism
\begin{equation}\label{A-A-iso}
A_{j}(r/r^{2})A_{i}\cong S_{j}
\otimes_{D_{j}} e^{j}_{11}(r/r^{2})e^{i}_{11}
\otimes_{ D_{i}}S_{i}^{op}.
\end{equation}

 Denote by $Arrow_{\Delta_A}(i,j)$ the $k$-linear space
generated by the set of  arrows from $i$ to $j$ in the natural quiver $\Delta_A$.
Then, $t_{ij}=\dim_kArrow_{\Delta_A}(i,j)$ is the minimal
number of generator of
$A_{j}(r/r^{2})A_{i}$ as $A_{j}$-$A_{i}$-module. Therefore
there is a
surjective homomorphism of $A_{j}$-$A_{i}$-bimodules
$$ A_{j}\otimes_kArrow_{\Delta_A}(i,j)\otimes_k A_{i}
\twoheadrightarrow  A_j(r/r^2)A_i.$$
In particular, by taking $M=A_{j}(r/r^{2})A_{i}$, we have a surjective
$A_{j}$-$A_{i}$-bimodule  homomorphism
\begin{equation}
\label{A-A-surj} A_{j}\otimes_kArrow_{\Delta_A}(i,j)\otimes_k A_{i}
\twoheadrightarrow  S_{j}\otimes _{D_{j}} e^{j}_{11}(r/r^{2})e^{i}_{11}
\otimes_{D_{i}}S_{i}^{op}.
\end{equation}
Since $A_i\cong S_{i}\otimes_{D_{i}}S_{i}^{op}$ as $A_{i}$-$A_{i}$-bimodule,
by applying the exactor functors  $\Hom_{A_{j}}(S_{j}, -)$ and
$\Hom_{A_{i}}(S_{i}^{op}, -)$ consecutively we get a surjective map of
$D_{j}$-$D_{i}$-bimodules
\begin{equation} \label{D-D-surj}S_{j}^{op}\otimes_{k}Arrow_{\Delta_A}(i,j)
\otimes_k S_{i}
\twoheadrightarrow   e^{j}_{11}(r/r^{2})e^{i}_{11}.
\end{equation}

In the definition of the diagram $D_{A}$ of $A$, we have $V_{i}\cong
(r/r^{2})e^{i}_{11}$ as left $A$-modules. Then by (a) and (b) above,  $
h_{ij}=\dim_{D_{j}}(e^{j}_{11}(r/r^{2})e^{i}_{11})$. Note that
$S_{j}\otimes _{D_{j}}
e^{j}_{11}(r/r^{2})e^{i}_{11}\otimes_{D_{i}}S_{i}^{op}$ can be
generated by $h_{ij}$ many generators as an
$A_{j}$-$A_{i}$-bimodule. Then \eqref{A-A-surj} implies that
$t_{ij}\leq h_{ij}$.

Since $A$ is Artinian, we can compare the $D_{j}$-dimensions from
\eqref{D-D-surj} to get
\[  h_{ij} \leq n_{i}n_{j}\dim_{k}(D_{i})t_{ij}.\]
In conclusion, we have
\begin{lemma}\label{prop5.2}
For an Artinian $k$-algebra $A$, we have
\[t_{ij}\leq h_{ij}\leq (n_in_jdim_kD_i)t_{ij}.\]
\end{lemma}

Similar to the setup in Section 4,  let $ \varepsilon_{i} \in A$ be idempotents that are inverse images of
 $e^{i}_{11}$ such that $\{ \varepsilon_{1}, \dots,\varepsilon_{s}\} $
 is a complete set of non-isomorphic
primitive idempotents of $A$. Set $ \varepsilon=\sum_{i=1}^{s}\varepsilon_{i}$.
Then $ B=\End_{A}(A\varepsilon, A\varepsilon)\cong \varepsilon A \varepsilon $
 is the basic algebra of $A$
with radical $r_{B}=\varepsilon r_{A} \varepsilon $. We know that
the diagram $D_B$ of $B$ is the same as that of $A$. In this case
$h_{ij}=\dim_{D_{j}}\varepsilon_{j}r_{B}/r_{B}^{2}\varepsilon_{i}$.
The natural quiver $ \Delta_{B}$ of $B$ has the number  of
arrows $|Arrow_{\Delta_{B}}(i,j)|$ from $i$ to $j$ equal to the
minimal number of generators of
$\varepsilon_{j}r_{B}/r_{B}^{2}\varepsilon_{i}$ as $
D_{j}$-$D_{i}$-bimodules. Note that
$\varepsilon_{j}r_{B}/r_{B}^{2}\varepsilon_{i} \cong
e_{11}^{j}r_{A}/r_{A}^{2}e_{11}^{i}$ as $D_{j}$-$D_{i}$-bimodules.
Then $S_{j}\otimes_{D_{j}}
e_{11}^{j}r_{A}/r_{A}^{2}e_{11}^{2}\otimes_{D_{i}}S^{op}_{i}$ can
be generated by $|Arrow_{\Delta_{B}}(i,j)|$ many elements as an
$A_{j}$-$A_{i}$-bimodule. We therefore get
\begin{equation} |Arrow_{\Delta_{A}}(i,j)|\leq |Arrow_{\Delta_{B}}(i,j)|.
\end{equation}

 For the basic algebra $B$, the system
 $\{ \varepsilon_{j}r_{B}/r_{B}^{2}\varepsilon_{i}\;
 |\; i,j =1,\dots, s\} $
 together with
 $ \{ D_{i}\; |\; i=1, \dots, s\}$ defines a $k$-species~\cite{R} .
 If all $D_{i}$ are finite dimensional over $k$,
 then the system is a modulation of the valued quiver for the algebra $A$ \cite{DR}.

\begin{theorem}\label{th5.3}
For an Artinian algebra $A$ with radical $r_A$, let
$\{S_1,\cdots,S_s\}$ be the complete set of all
non-isomorphic irreducible $A$-modules and $D_i=\End_A(S_i)$ for
any $i=1,\cdots,s$. Let $B$ be the basic algebra of $A$
with radical $r_{B}$ such that $
B/r_{B}=\prod_{i=1}^{s}D_{i}$. Let  $t^{A}_{ij}$ (resp. $t_{ij}^B$) and
$h^{A}_{ij}$ (resp. $h_{ij}^B$)  be the numbers of arrows from $i$ to $j$ in
the natural quiver $\Delta_A$ (resp. $\Delta_B$) and the diagram $D_A$ (resp. $D_B$)
respectively. For any $i,j=1,\cdots,s$, we have
\begin{itemize}
\item[(i)] $t^{A}_{ij}\leq \lceil
\frac{t^{B}_{ij}}{n_{i}n_{j}}\rceil\leq t^{B}_{ij}$;
\item[(ii)]
$ t^{A}_{ij}\leq \lceil \frac{h^{A}_{ij}}{n_{i}n_{j}}\rceil \leq
t^{A}_{ij}\dim_{k}D_{i}.$
\end{itemize}
\end{theorem}
\begin{proof} It follows from Proposition 5.1 that  $h^A_{ij}=h_{ij}^B$.
If $M$ is a $D_j$-$D_i$-bimodule which can be
generated by $m$ many elements as a $D_j$-$D_i$-bimodule, then there is a surjective map
$(D_j\otimes_k D_i)^{\oplus m}\rightarrow M$ as
$D_j$-$D_i$-bimodules. Thus we have a surjective map $S_j\otimes_{D_j}
(D_j\otimes_k D_i)^{\oplus m} \otimes_{D_i}S_i^{op}\rightarrow
S_j\otimes_{D_j} M \otimes_{D_i}S_i^{op} $ of $
A_j$-$A_i$-bimodules.  Let $q=\lceil \frac{m}{n_jn_i}\rceil$. Hence
we have a surjective map $ (S_j\otimes_{k}S_i^{op})^{\oplus
qn_jn_i}\rightarrow (S_j\otimes_{k}S_i^{op})^{\oplus m}\cong
S_j\otimes_{D_j} (D_j\otimes_k D_i)^{\oplus m}
\otimes_{D_i}S_i^{op}$ as $A_j$-$A_i$-bimodules.
 Since  $(S_j\otimes_{k}S_i^{op})^{\oplus n_jn_i}\cong
 S_j^{\oplus n_j}\otimes_{k}(S_i^{op})^{\oplus n_i}$ as
 $A_j$-$A_i$-bimodules, which can be generated
 by one element as an $A_j$-$A_i$-bimodule,
 $(S_j\otimes_{k}S_i^{op})^{\oplus
qn_jn_i}$ can be generated by $q$ elements as $A_j$-$A_i$-bimodule. Therefore $
S_j\otimes_{D_j}M\otimes_{D_i} S_i^{op}$ can be generated by $q$
elements as an $A_j$-$A_i$-module.

Applying the above argument to $M=e_{11}^j (r_A/r^2_A)e_{11}^i$,
where $M$ can be generated by $t^B_{ij}$ many elements as a
$D_j$-$D_i$-bimodule, we conclude that $A_j(r_A/r^2_A)A_i$ can be generated by
$\lceil \frac{t^{B}_{ij}}{n_jn_i}\rceil$ many elements as
$A_j$-$A_i$-bimodules. This shows that $ t_{ij}^A \leq \lceil
\frac{t^{B}_{ij}}{n_jn_i}\rceil \leq t^B_{ij}$, where the first inequality follows from the fact that $t^A_{ij}$ is the minimum.

To show (ii), we note that $t^B_{ij}\leq h^B_{ij}=h^A_{ij}\leq n_i n_j\dim_kD_i t_{ij}^A$ by Lemma~\ref{prop5.2}.
By (i), we have
$ t^A_{ij}\leq \lceil \frac{t^B_{ij}}{n_in_j}\rceil \leq  \lceil \frac{h^A_{ij}}{n_in_j}\rceil \leq
 \lceil \frac{n_i n_j\dim_kD_i t_{ij}^A}{n_in_j}\rceil=t_{ij}^A\dim_kD_i .$
\end{proof}

The relation (ii) implies that $t_{ij}$ is never larger than $h_{ij}$,
and $t_{ij}\not=0$ if and only if $h_{ij}\not=0$ for any
$i,j\in(\Delta_A)_0$, which gives the relation between the natural quiver and the diagram.

The relation (i) gives the relation between the natural quiver and the
Ext-quiver of an Artinian algebra over an arbitrary field $k$. Note
that the formula $t_{ij}=\lceil \frac{m_{ij}}{n_in_j} \rceil$ in
Theorem \ref{th2.2} holds only in the case that $A$ is $k$-splitting, since the diagram $D_A$ and the ext-quiver
$\Gamma_A$ are the same when $A$ is splitting over $k$.

 In general, when $A$ is basic (i.e, $A_i=D_i$ is a division ring) without the splitting
condition over $k$, the natural
quiver $\Delta_A$ will not be the same as either $\Gamma_A$ nor $D_A$. In this case, the discussion before
Theorem~\ref{th5.3} relates $\Delta_A$ and $D_A$ by the species $\{D_i, \varepsilon_j
r_A/r_A^2\varepsilon_i\}$ \cite{R} as   follows.

\begin{itemize}
\item[(i)]  $h_{ij}=\dim_{D_j}(A_j r_A/r_A^2 A_i)$ and $ t_{ij}$ is minimal number of generators of $A_j
r_A/r_A^2 A_i$ as $ A_j$-$A_i$-bimodule;
\item[(ii)]
if in addition that $A$ is splitting over $k$, then $h_{ij}=t_{ij}=m_{ij}$,
which means $D_A=\Delta_A=\Gamma_A$.
\end{itemize}

\begin{example} Let $ E/k$ be a field extension. Consider the algebra
$A=\begin{bmatrix}E &M \\
0& E
\end{bmatrix}$
 is a basic algebra over $k$ for any $E$-$E$-bimodule $M$. With different choices of
 $M$ one can have either $ h_{ij}=t_{ij}$ (taking $ M=E$) or $h_{ij}$ and $t_{ij}$ to be quite different
 (taking $M=E\otimes_k E$).
 \end{example}




\section { Hereditary algebras as generalized path algebras }

We have mentioned the result in \cite{DK} that any finite dimensional hereditary
algebra $A$ is  isomorphic to the tensor algebra $T(A/r,
r/r^2)$ if the quotient algebra $A/r$ is separable for the radical $r$.
 As a comparison, in this section we will show  that an Artinian hereditary algebra  $A$ is
always  isomorphic to $k(\Delta_A,\mathcal A)$ if $A$ is of
Gabriel-type with admissible defining ideal.

It is proved in (\cite{DK}, Corollary 3.7.3) that the diagram $D_A$ has
no cycles if  $A$ is a finite dimensional hereditary algebra. In fact this is true for Artinian hereditary algebras.

\begin{lemma}\label{newlemm6.1}\cite{HGK}
If a ring $A$ is hereditary, then any nonzero homomorphism $\varphi: P_1\rightarrow P_2$ of indecomposable projective $A$-modules is a monomorphism.
\end{lemma}

\begin{lemma}\label{newlemm6.2}
If an Artinian algebra $A$ is hereditary, then its diagram $D_A$ has
no cycles.
\end{lemma}
\begin{proof}
For any vertices $i$ and $j$, if there is an arrow $\sigma$ from $j$ to $i$ in $D_A$, then there exists a non-zero homomorphism $f_{\sigma}: P_j\rightarrow P_i$ of projective covers irreducible modules $S_{j}$ and $S_{i}$ such that $ f_{\sigma}(P_{j})\subseteq r P_{i}$. By Lemma \ref{newlemm6.1}, $f_\sigma$ is always a monomorphism, but not onto.

Suppose $D_A$ has a cycle $\sigma_1\sigma_2\cdots\sigma_s$ with both tail and head at the vertex $i$. Then, $f=f_{\sigma_1}f_{\sigma_2}\cdots f_{\sigma_s}$ is a monomorphism from $P_i$ to $P_i$ since each $f_{\sigma_i}$ is a monomorphism for any $i$. It is not isomorphic, i.e. $f(P_i)\subsetneqq P_i$. Moreover, it follows the infinite sequence:
\[P_i\varsupsetneqq f(P_i)\varsupsetneqq f^2(P_i)\varsupsetneqq \cdots \varsupsetneqq f^l(P_i)\varsupsetneqq\cdots.
\]
Note that $P_i$ is isomorphic to a left ideal of $A$, thus the above
contradicts to the fact $A$ is Artinian.
\end{proof}

\begin{proposition}\label{prop6.1}$\;$ Let $A$ be a hereditary
  Artinian  algebra.
Then the natural quiver $\Delta_A$ of $A$ is finite and acyclic.
\end{proposition}
\begin{proof} Finiteness of $\Delta_A$ is trivial. According
to the relation between  $\Delta_A$ and $D_A$ in
 Theorem \ref{th5.3}, $\Delta_A$ is acyclic if and only if $D_A$
  is acyclic.  By Lemma \ref{newlemm6.2}, $D_A$ is acyclic.
  Hence, $\Delta_A$  is acyclic.
\end{proof}

By definition, normal generalized path algebras
 can be thought as a special class of tensor algebras, which are always hereditary due to \cite{A}.
The following main result in this section can be thought as a partial converse of this statement.
\begin{proposition} \label{prop6.3}  Let $\pi: B\rightarrow A$ be a surjective homomorphism of two
hereditary algebras $A$ and $B$ such that $I=\ker (\pi)\subseteq
\rad(B)^2$. If $B$ is Artinian, then $I=0$ and $\pi$ is an
isomorphism.
\end{proposition}
\begin{proof} Let $r_B=\rad(B)$. Since $B$ is Artinian, then $ r_B^n=0 $ for some $n$.  We have $A\cong B/I$ and
$r_B^s\subset
I\subset r_B^2$ for some positive integer $s$. It is enough to prove
that $I=0$.

Let $R=r_B/I$. By induction on $k$, we will prove that for any $k\geq
0$, \begin{equation}\label{isom6}
 R^{k}/R^{k+1}\cong r_B^{k}/r_B^{k+1}
\end{equation}
as $A$-modules. Here the $A$-module structure of  $r_B^{k}/r_B^{k+1}$
can be induced naturally from the $B$-module
structure of $r_B^{k}/r_B^{k+1}$ since $I(r_B^{k}/r_B^{k+1})=0$.

When $k=1$, we have $R/R^2=(r_B/I)/(r_B/I)^2=(r_B/I)/(r_B^2/I)\cong r_B/r_B^2$
since $I\subset r_B^2$.

Suppose that (\ref{isom6}) holds for $k-1\geq 0$, that is,
$R^{k-1}/R^{k}\cong r_B^{k-1}/r_B^{k}$ as $A$-modules. For the case of
 $k$, we discuss as follows.

We first note that for any projective $B$-module $Q$,
$Q/IQ=A\otimes_B Q$ is a projective $A$-module. In particular,
for any $A$-module $M$, if $P_B(M)$ is projective cover of $M$ as $B$-module, then $ A\otimes_B P_B(M)$
is the projective cover of $M/IM$ as $A$-module. Now
let $\bar P=P_A(R^{k-1}/R^k)$  be the projective covers of
$R^{k-1}/R^k$  as $A$-modules,  and let $Q=P_B(r_B^{k-1}/r_B^k)$ be the
projective cover of $r_B^{k-1}/r_B^k$ as $B$-module.  Thus by induction assumption, $r_B^{k-1}/r_B^k\cong
R^{k-1}/R^{k}$ as $ A$-module. Then  we
have $\bar{P}\cong Q/IQ$ as $A$-modules  (by \cite{Lam} (pp. 363-364)).
 Let $ \pi: Q\rightarrow
Q/IQ$ being the quotient map, we have
\begin{eqnarray*}
R\bar{P}&\cong &\pi(r_B)\pi(Q)=\pi(r_BQ)\cong r_BQ/IQ ,\\
R^2\bar{P}&\cong &\pi(r_B^2)\pi(Q)=\pi(r_B^2Q)\cong r_B^2Q/IQ.
\end{eqnarray*}

Since $R=r_B/I\cong \rad A$,
 we have $R^k=RR^{k-1}\cong (\rad A)R^{k-1}=\rad
 R^{k-1}$, thus $\bar P=P_A(R^{k-1}/R^k)=P_A(R^{k-1}/\rad
 R^{k-1})=P_A(R^{k-1})$. Similarly, $r_B^k=\rad
 r_B^{k-1}$ and then, $Q=P_B(r_B^{k-1}/r_B^k)=P_B(r_B^{k-1})$.
 Here we are using the condition $I \subseteq r_B^2$.

Since $A$ is hereditary, $R$ is projective as $A$-module and
we have $P_A(R^{k-1})=R^{k-1}$. Similarly
 $P(r_B^{k-1})=r_B^{k-1}$ since $B$ is hereditary. Furthermore,
 $R^{k}\cong R\bar{P}$, $R^{k+1}\cong R^2\bar{P}$ as $ A$-modules and
 $r_B^k=r_BQ$ and $r_B^{k+1}=r_B^2Q$ as $B$-modules.
 Therefore we have the following isomorphisms of $A$-modules.
\[
 R^k/R^{k+1}=(R\bar P)/(R^2\bar P)
  \cong (r_BQ/IQ)/(r_B^2Q/IQ)\cong r_B^{k}/r_B^{k+1}.\]
If one tracks the isomorphism, one would find that the above
isomorphism is actually induced from the quotient map
$\pi:B \rightarrow A$ since $\bar{P}$ and $Q$
are subsets of $ A$ and $B$ respectively (both
algebras are hereditary).

Since $B$ is Artinian,
 we have $r_B^m=0$ for some $m$. If
 $ I\neq 0$, there exists $p\geq 1$ minimal such that $ I\not
 \subseteq r_B^p$. Then take $ x\in I\setminus r_B^p$. It means $ x\in r_B^{p-1}$
 and  $ \pi(x)=0 \in R^{p-1}$. But $\pi: r_B^{p-1}/r_B^{p}\rightarrow
 R^{p-1}/R^{p}$ is an isomorphism. Hence, the image of $ x$ in $
 r_B^{p-1}/r_B^p$ has to be zero, i.e., $ x \in r_B^{p}$ which contradicts
 the choice of $x$. Thus we must have $ I=0$. This proves the result.
 \end{proof}

\begin{theorem}\label{theorem6.3}
Let $A$ be a hereditary Artinian algebra splitting over radical such that the surjective
homomorphism $\pi: k(\Delta_A, \mathcal A)\rightarrow A$ in
Theorem \ref{th3.3} possesses the kernel $\ ker(\pi)\subseteq J^2$.
Then, $\pi$ is an isomorphism.
\end{theorem}
\begin{proof}
By Proposition \ref{prop6.1}, $\Delta_A$ is finite and acyclic. Then, $k(\Delta_A, \mathcal A)$ is finite-dimensional and $\rad k(\Delta_A, \mathcal A)=J$ the ideal generated by all $\mathcal A$-paths of length $1$. Thus, the condition in Proposition \ref{prop6.3} is satisfied for the natural homomorphism from $k(\Delta_A, \mathcal A)$ to $A$. It follows that $A\cong k(\Delta_A, \mathcal A)$.
\end{proof}

In Theorem \ref{theorem6.3}, $ A/r_A$ is not usually a
separable algebra. Hence, this result can be thought as an improvement of  \cite[Theorem 8.5.4]{DK}.

\end{document}